\documentclass{article}

\usepackage{amssymb}
\usepackage{hyperref}
\usepackage{amsmath}
\usepackage{graphicx}

\usepackage{cases}
\usepackage{subfigure}
\usepackage{amsthm}
\usepackage[german,english]{babel}
\usepackage{color}
\usepackage{epstopdf}

\usepackage{mathabx}

\numberwithin{equation}{section}

\newtheorem*{assumption*}{Assumption}
\newtheorem{assumption}{Assumption}

\newtheorem{remark}{Remark}
\numberwithin{remark}{section}
\newtheorem*{remark*}{Remark}

\newtheorem{lemma}{Lemma}
\numberwithin{lemma}{section}
\newtheorem{theorem}{Theorem}

\let \b=\boldsymbol

\begin{document}
\title{Discrete Green functions of the SDFEM on Shishkin triangular meshes}

\author{%
Jin Zhang\thanks{Email: jinzhangalex@hotmail.com}
\footnote{School of Mathematical Sciences, Shandong Normal University,
Jinan 250014, China}}
\maketitle

\begin{abstract}
We present bounds   of  discrete Green functions  in the  energy norm for the standard (or modified) streamline
diffusion finite element method (SDFEM) on Shishkin triangular meshes.
\end{abstract}
\section{Problem}
We consider the singularly perturbed boundary value problem
 \begin{equation}\label{eq:model problem}
 \begin{array}{rcl}
-\varepsilon\Delta u+\boldsymbol{b}\cdot \nabla u+cu=f & \mbox{in}& \Omega=(0,1)^{2},\\
 u=0 & \mbox{on}& \partial\Omega ,
 \end{array}
 \end{equation}
where $\varepsilon\ll |\boldsymbol{b}|$ is a small positive parameter, 
$\boldsymbol{b}=(b_{1},b_{2})^{T}$ is a constant vector with $b_1>0,b_2>0$ and $c>0$  is constant. It is also
assumed that $f$ is sufficiently smooth. The solution of \eqref{eq:model problem} typically has
two exponential layers of width $O(\varepsilon\ln(1/\varepsilon))$ at the sides
$x=1$ and $y=1$ of $\Omega$.

\section{The SDFEM on Shishkin meshes}

\subsection{Shishkin meshes }
When discretizing \eqref{eq:model problem}, we use \textit{Shishkin} meshes, which are piecewise uniform. See \cite{Roos:1998-Layer,Roo1Sty2Tob3:2008-Robust,Linb:2003-Layer} for a detailed discussion of their properties and applications.

First, we define two mesh transition parameters, which are to be used to specify the mesh changes from coarse to fine in $x-$ and $y-$direction,
\begin{equation*}
\lambda_{x}:=\min\left\{ \frac{1}{2},\rho\frac{\varepsilon}{\beta_{1}}\ln N \right\} \quad \mbox{and} \quad
\lambda_{y}:=\min\left\{
\frac{1}{2},\rho\frac{\varepsilon}{\beta_{2}}\ln N \right\},
\end{equation*}
where $\beta_1$ and $\beta_2$ are defined as in \cite[Assumption 2.1]{Linb1Styn2:2001-SDFEM}.  For technical reasons, we set $\rho=2.5$ in our analysis which is the same as ones in \cite{Zhang:2003-Finite} and \cite{Styn1Tobi2:2003-SDFEM}. The domain $\Omega$ is dissected into four parts as $\bar{\Omega}=\Omega_{s}\cup\Omega_{x}\cup\Omega_{y}\cup\Omega_{xy}$ (see Figure \ref{Shishkin mesh}), where
\begin{align*}
&\Omega_{s}:=\left[0,1-\lambda_{x}\right]\times\left[0,1-\lambda_{y}\right],&&
\Omega_{x}:=\left[ 1-\lambda_{x},1 \right]\times\left[0,1-\lambda_{y}\right],\\
&\Omega_{y}:=\left[0,1-\lambda_{x}\right]\times\left[1-\lambda_{y},1 \right],&&
\Omega_{xy}:=\left[ 1-\lambda_{x},1 \right]\times\left[1-\lambda_{y},1 \right].
\end{align*}
\begin{assumption}
We assume  that $\varepsilon\le N^{-1}$, as is generally the case in practice. Furthermore we assume that
$\lambda_{x}=\rho\varepsilon\beta^{-1}_{1}\ln N$ and $\lambda_{y}=\rho\varepsilon\beta^{-1}_{2}\ln N$ as otherwise $N^{-1}$ is exponentially small compared with $\varepsilon$.
\end{assumption}

\begin{figure}
\begin{minipage}[t]{0.5\linewidth}
\centering
\includegraphics[width=2.5in]{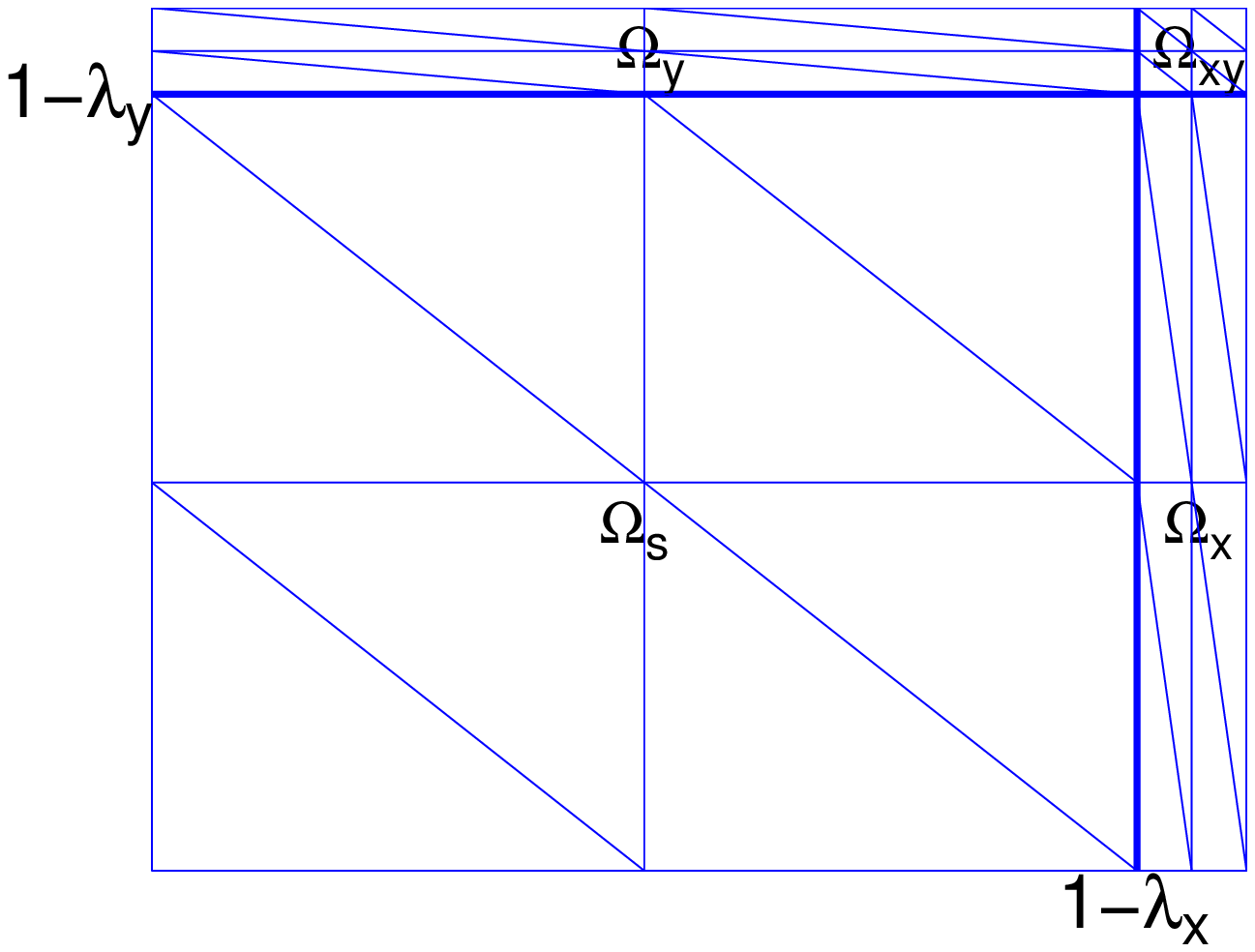}
\caption{Dissection of $\Omega$ and triangulation $\mathcal{T}_{N}$.}
\label{Shishkin mesh}
\end{minipage}%
\begin{minipage}[t]{0.5\linewidth}
\centering
\includegraphics[width=2.5in]{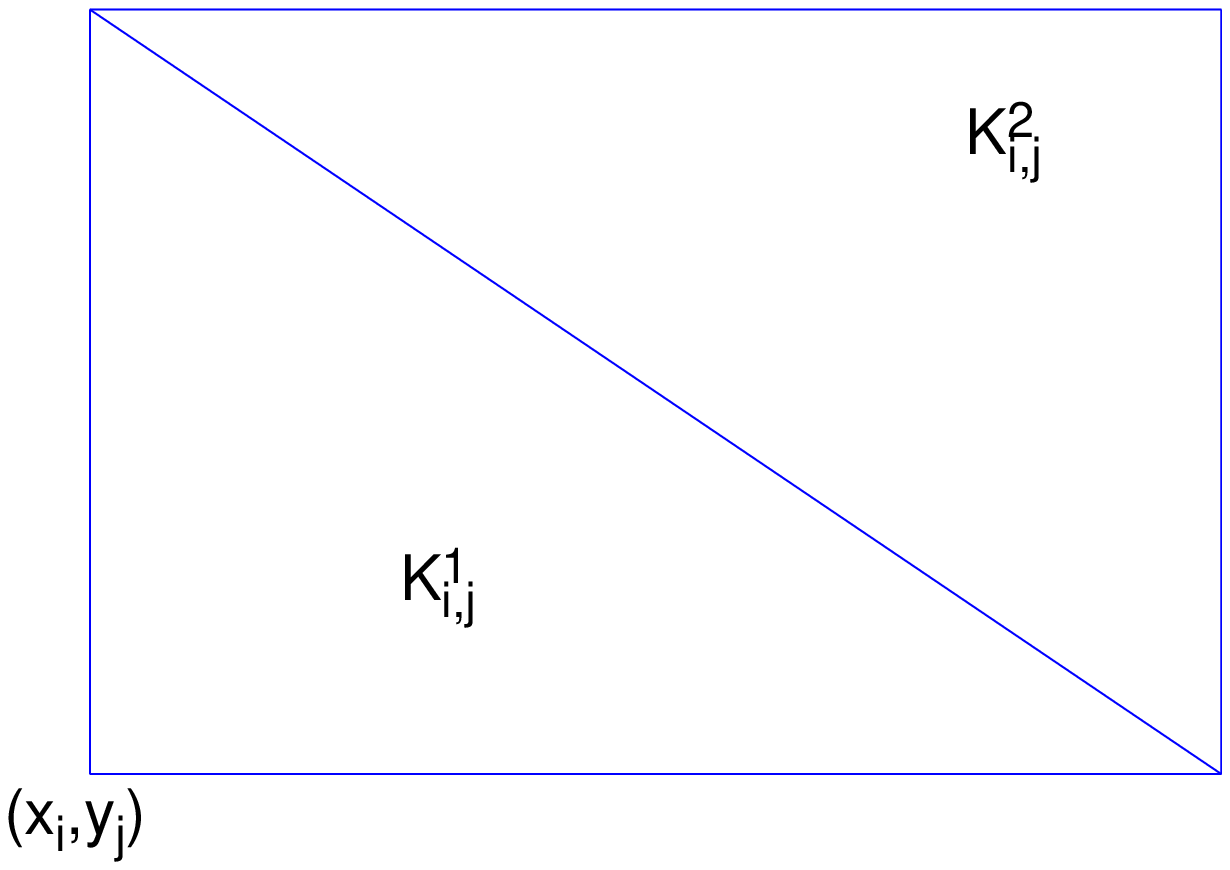}
\caption{$K^{1}_{i,j}$ and $K^{2}_{i,j}$}
\label{fig:code of mesh}
\end{minipage}
\end{figure}

\par
Next, we define the set of mesh points $\left\{ (x_{i},y_{j})\in\Omega:i,\,j=0,\,\cdots,\,N  \right\}$
\begin{numcases}{x_{i}=}
2i(1-\lambda_{x})/N &\text{for $i=0,\,\cdots,\,N/2$}, \nonumber\\
1-2(N-i)\lambda_{x}/N &\text{for $i=N/2+1,\,\cdots,\,N$}\nonumber
\end{numcases}
and
\begin{numcases}{y_{j}=}
2j(1-\lambda_{y})/N &\text{for $j=0,\,\cdots,\,N/2$}, \nonumber\\
1-2(N-j)\lambda_{y}/N &\text{for $j=N/2+1,\,\cdots,\,N$}.\nonumber
\end{numcases}
By drawing lines through these mesh points parallel to the $x$-axis and $y$-axis the domain $\Omega$ is partitioned into rectangles. Each rectangle is divided into two triangles by drawing the diagonal which runs from $(x_i,y_{j+1})$ to $(x_{i+1},y_{j})$. This yields a triangulation of $\Omega$  denoted by $\mathcal{T}_{N}$(see Fig. \ref{Shishkin mesh}).
The mesh sizes $h_{x,i}:=x_{i+1}-x_{i}$ and $h_{y,j}:=y_{j+1}-y_{j}$ satisfy
\begin{numcases}{h_{x,i}=}
H_{x}:=\frac{1-\lambda_{x}}{N/2}  &\text{for $i=0,\,\cdots,\,N/2-1$}, \nonumber\\
h_{x}:=\frac{\lambda_{x}}{N/2}  &\text{for $i=N/2,\,\cdots,\,N-1$} \nonumber
\end{numcases}
and
\begin{numcases}{h_{y,j}=}
H_{y}:=\frac{1-\lambda_{y}}{N/2}  &\text{for $j=0,\,\cdots,\,N/2-1$},\nonumber \\
h_{y}:=\frac{\lambda_{y}}{N/2}  &\text{for $j=N/2,\,\cdots,\,N-1$}.\nonumber
\end{numcases}
The mesh sizes $h_{x,i}$ and $h_{y,j}$ satisfy
\begin{equation*}
N^{-1}\le H_{x},H_{y} \le 2N^{-1} \quad \mbox{and} \quad
C_{1}\varepsilon N^{-1}\ln N \le h_{x},h_{y}\le C_{2}\varepsilon N^{-1}\ln N.
\end{equation*}

For convenience, we shall use some notations: $K^1_{i,j}$ for the mesh triangle with vertices $(x_i,y_j)$, $(x_{i+1},y_j)$ and $(x_i,y_{j+1})$;  $K^2_{i,j}$  for the mesh triangle with vertices
$(x_i,y_{j+1})$, $(x_{i+1},y_j)$ and $(x_{i+1},y_{j+1})$ (see Fig. \ref{fig:code of mesh}); $\tau$ or $K$ for a generic mesh triangle.

\subsection{The streamline diffusion finite element method}
On the above Shishkin meshes we define a $C^0$ linear finite element space
\begin{equation*}
V^{N}:=\{v^{N}\in C(\bar{\Omega}):v^{N}|_{\partial\Omega}=0
 \text{ and $v^{N}|_{K}\in P_{1}(K)$, }
  \forall K\in \mathcal{T}_{N} \}.
\end{equation*}

Now we are in a position to state the standard SDFEM for \eqref{eq:model problem} which reads: 
\begin{equation}\label{eq:SDFEM-I}
\left\{
\begin{split}
&\text{Find $u^{N}\in V^{N}$ such that for all $v^{N}\in V^{N}$},\\
&a_{SD}(u^{N},v^{N})=(f,v^{N})+\underset{K\subset\Omega}\sum(f,\delta_{K}\boldsymbol{b}\cdot\nabla v^{N})_{K},
\end{split}
\right.
\end{equation}
where
\begin{align*}
a_{SD}(u^{N},v^{N})=&\varepsilon (\nabla u^{N},\nabla v^{N})+(\boldsymbol{b}\cdot\nabla u^{N},v^{N})
+(cu^{N},v^{N})\\
&+\sum_{K\subset\Omega}(-\varepsilon\Delta u^{N}+\boldsymbol{b}\cdot\nabla u^{N}+cu^{N},\delta_{K}\boldsymbol{b}\cdot\nabla v^{N})_{K}.
\end{align*}
Note that $\Delta (u^N|_K) =0$   for $u^N\vert_K\in P_1(K)$. Following usual practice
\cite{Roo1Sty2Tob3:2008-Robust},  the parameter $\delta_{K}:=\delta|_K$ is defined as follows
\begin{equation}\label{eq: delta-K}
\delta:=\delta(x,y)=
\left\{
\begin{matrix}
C^{\ast}N^{-1}
&\text{if $(x,y)\in \Omega_{s}$},\\
0&\text{otherwise},
\end{matrix}
\right.
\end{equation}
where $C^{\ast}$ is referred to \cite[Lemma 3.25]{Roo1Sty2Tob3:2008-Robust}.

\par
We set
\begin{equation*}
    b:=\sqrt{b^{2}_{1}+b^{2}_{2}},\quad
    \boldsymbol{\beta}:=\genfrac(){0cm}{0}{b_{1}}{b_{2}}/b,\quad
    \boldsymbol{\eta}:=\genfrac(){0cm}{0}{-b_{2}}{b_{1}}/b
     \quad\mbox{and}\quad
    v_{\zeta}:=\boldsymbol{\zeta}^{T}\nabla v
\end{equation*}
for any vector $\boldsymbol{\zeta}$ of unit length. By an easy calculation one shows that
\begin{equation*}
(\nabla w,\nabla v)=(w_{\beta},v_{\beta})+(w_{\eta},v_{\eta}).
\end{equation*}
We rewrite \eqref{eq:SDFEM-I} as
\begin{align*}
&\varepsilon(u^N_{\beta},v^{N}_{\beta})+\varepsilon(u^N_{\eta},v^{N}_{\eta})
+(bu^N_{\beta}+u^N,v^{N})
+\sum_{ K\subset\Omega } (bu^N_{\beta}+cu^N,\delta_K b v^{N}_{\beta})_K\\
&=(f,v^{N})+\sum_{ K\subset\Omega } (f,\delta_K b v^{N}_{\beta})_K.
\end{align*}

For technical reasons in the later analysis, we increase the crosswind diffusion(see \cite{Joh1Sch2Wah3:1987-Crosswind,Niijima:1990-Pointwise,Roo1Sty2Tob3:1996-Numerical}) by replacing $\varepsilon(u^N_{\eta},v^{N}_{\eta})$ by
$\hat{\varepsilon}(u^N_{\eta},v^{N}_{\eta})$ where $\varepsilon \le \hat{\varepsilon} \le N^{-1}$ and $\hat{\varepsilon}$ is constant on each of subdomains including $\Omega_s,\Omega\setminus\Omega_s$. For convenience, we denote $\hat{\varepsilon}|_{\Omega_s}$by $\hat{\varepsilon}_s$.
We will consider two cases: (1) $\hat{\varepsilon}=\varepsilon$; (2) $\hat{\varepsilon}$ defined as in  \cite[pg 463]{Linb1Styn2:2001-SDFEM}, that is
$$
\hat{\varepsilon}=
\left\{
\begin{matrix}
\tilde{\varepsilon} & \text{if $\b{x}\in \Omega_s$},\\
\varepsilon &  \text{if $\b{x}\in \Omega\setminus\Omega_s$}
\end{matrix}
\right.
$$
where $\tilde{\varepsilon}:=\max \left( \varepsilon, N^{-3/2} \right)$.

We now state a streamline diffusion method with artificial crosswind diffusion (ACD), also called a modified SDFEM:
\begin{equation*}
\left\{
\begin{split}
&\text{Find $u^N\in V^{N}$ such that for all $v^{N}\in V^{N}$}\\
&a_{MSD}(u^N,v^{N})=(f,v^{N}+\delta bv^{N}_{\beta}),
\end{split}
\right.
\end{equation*}
with
\begin{equation}\label{eq:SDFEM with ACD}
\begin{split}
a_{MSD}(u^N,v^{N}):=&\varepsilon(u^N_{\beta},v^{N}_{\beta})+\hat{\varepsilon}(u^N_{\eta},v^{N}_{\eta})
+(bu^N_{\beta}+u^N,v^{N})\\
&+\sum_{  K\subset\Omega  } (bu^N_{\beta}+u^N,\delta_K b v^{N}_{\beta})_K.
\end{split}
\end{equation}
Finally, we define a special energy norm associated with $a_{MSD}(\cdot,\cdot)$:
\begin{equation*}
\Vert v^N \Vert^{2}_{MSD}:=
\varepsilon \Vert  v^N_{\beta}\Vert^2+
\hat{\varepsilon} \Vert v^N_{\eta} \Vert^2+\Vert v^N \Vert^2+\sum_{  K\subset\Omega  } \delta_K \Vert b v^N_{\beta} \Vert^2_K,\quad \forall v^N\in V^N.
\end{equation*}
For brevity, we often write $\varepsilon(u^N_{\beta},v^{N}_{\beta})+\sum\limits_{  K\subset\Omega  }(bu^N_{\beta}, \delta_{K} \;bv^N_{\beta})_K$ as
 $(\varepsilon+b^2\delta)(u^N_{\beta},v^{N}_{\beta})$.

\section{The discrete Green function}
Let $\boldsymbol{x}^{\ast}$ be a mesh node in $\Omega$. The discrete Green's function $G\in V^{N}$ associated with $\boldsymbol{x}^{\ast}$ is defined by
\begin{equation}\label{eq:discrete Green function}
a_{MSD}(v^{N},G)=v^{N}(\boldsymbol{x}^{\ast})\quad \forall v^{N}\in V^{N}.
\end{equation}

\par\noindent
The weight function $\omega$ is defined by
\begin{equation*}
\omega(\boldsymbol{x}):=
g\left(\frac{(\boldsymbol{x}-\boldsymbol{x}^{\ast})\cdot\boldsymbol{\beta}}{\sigma_{\beta}}\right)
g\left(\frac{(\boldsymbol{x}-\boldsymbol{x}^{\ast})\cdot\boldsymbol{\eta}}{\sigma_{\eta}}\right)
g\left(-\frac{(\boldsymbol{x}-\boldsymbol{x}^{\ast})\cdot\boldsymbol{\eta}}{\sigma_{\eta}}\right)
\end{equation*}
where
\begin{equation*}
g(r)=\frac{2}{1+e^{r}}\quad\text{ for $r\in(-\infty,\infty)$}.
\end{equation*}

Now,  we are to derive a global estimate on $G$ in the weighted energy norm 
\begin{align}
 \vvvert G \vvvert^{2}_{\omega}:= &\varepsilon \Vert \omega^{-1/2}G_{\beta} \Vert^{2}
+\hat{\varepsilon}\Vert \omega^{-1/2}G_{\eta} \Vert^{2}+\frac{b}{2}\Vert (\omega^{-1})^{1/2}_{\beta}G\Vert^{2}\label{eq:weighted norm}\\
&+c\Vert \omega^{-1/2}G \Vert^{2}
+\sum_K b^{2}\delta_K\Vert \omega^{-1/2}G_{\beta} \Vert^{2}_K\nonumber.
\end{align}

From \eqref{eq:SDFEM with ACD} and \eqref{eq:weighted norm}, we have
\begin{align}\label{eq:norm and bilinear form}
\vvvert G \vvvert^{2}_{\omega}
=&a_{MSD}(\omega^{-1}G,G)-\varepsilon((\omega^{-1})_{\beta}G,G_{\beta})
- \hat{\varepsilon} ((\omega^{-1})_{\eta}G,G_{\eta})\\
&
-\sum_K  ( b(\omega^{-1})_{\beta}G+c \omega^{-1}G,\delta_K \; bG_{\beta})_K
\nonumber.
\end{align}
Considering \eqref{eq:discrete Green function} we also have
\begin{equation}\label{eq:idea-weighted estimate}
\begin{split}
a_{MSD}(\omega^{-1}G,G)&=a_{MSD}(E,G)+a_{MSD}((\omega^{-1}G)^{I},G)\\
&=a_{MSD}(E,G)+(\omega^{-1}G)(\boldsymbol{x}^{\ast})
\end{split}
\end{equation}
where $E:=\omega^{-1}G-(\omega^{-1}G)^{I}$. 

%

\begin{lemma}\label{lem:1}
If 
\begin{equation}\label{eq:sigma-beta-eta-1}
\sigma_{\beta} \ge k(\varepsilon+\delta_M) \quad \text{and}\quad\sigma_{\eta}\ge k   \hat{\varepsilon }^{1/2}_M, 
\end{equation}
where $\delta_M:=\max_{\b{x}\in\Omega }\delta$ and $\hat{\varepsilon }_M:=\max_{\b{x}\in\Omega }\hat{\varepsilon }$, 
then for $k>1$ sufficiently
large and independent of $N$ and $\varepsilon$, we have
\begin{equation*}
a_{MSD}(\omega^{-1}G,G)\ge \frac{1}{4} \vvvert G  \vvvert^{2}_{\omega}.
\end{equation*}
\end{lemma}
\begin{proof}
H\"{o}lder inequalities, Cauchy inequalities and \cite[Lemma 4.1]{Linb1Styn2:2001-SDFEM} give
\begin{align*}
&(\varepsilon+b^{2}\delta)\left|((\omega^{-1})_{\beta}G,G_{\beta})\right|\\
\le&
C\varepsilon^{1/2} \sigma^{-1/2}_{\beta} \Vert  (\omega^{-1})_{\beta}^{1/2}G \Vert  \cdot\varepsilon^{1/2}  \Vert  \omega^{-1/2} G _{\beta} \Vert \\
&+\sum_K Cb\delta_K^{1/2}  \sigma^{-1/2}_{\beta}\Vert  (\omega^{-1})_{\beta}^{1/2}G \Vert_K\cdot  b\delta_K^{1/2}  \Vert  \omega^{-1/2} G _{\beta} \Vert_K\\
\le&
\frac{1}{2} \varepsilon  \Vert  \omega^{-1/2} G _{\beta} \Vert^2
+C \varepsilon  \sigma^{-1}_{\beta}  \Vert  (\omega^{-1})_{\beta}^{1/2}G \Vert^2+\frac{1}{4}\sum_K b^{2}\delta_K\Vert \omega^{-1/2}G_{\beta} \Vert^{2}_K\\
&+C \sum_K \delta_K   \sigma^{-1}_{\beta}  \Vert  (\omega^{-1})_{\beta}^{1/2}G \Vert^2 \\
\le&
\frac{1}{2} \varepsilon  \Vert  \omega^{-1/2} G _{\beta} \Vert^2+\frac{1}{4}\sum_K b^{2}\delta_K\Vert \omega^{-1/2}G_{\beta} \Vert^{2}_K+
C(\varepsilon +\delta )\sigma^{-1}_{\beta}  \Vert  (\omega^{-1})_{\beta}^{1/2}G \Vert^2.
\end{align*}
Similarly, we have
\begin{align*}
\hat{\varepsilon }\left|((\omega^{-1})_{\eta}G,G_{\eta})\right|
&\le
C\hat{\varepsilon }^{1/2}\sigma^{-1}_{\eta}\cdot
\Vert \omega^{-1/2}G \Vert
\cdot\hat{\varepsilon }^{1/2}\Vert \omega^{-1/2}G_{\eta} \Vert\\
&\le
\frac{1}{2}\hat{\varepsilon } \Vert \omega^{-1/2}G_{\eta} \Vert^2+\frac{1}{2}C\hat{\varepsilon } \sigma^{-2}_{\eta}\;\; c\Vert \omega^{-1/2}G \Vert^2
\end{align*}
and
\begin{align*}
\sum_K  ( c \omega^{-1}G,\delta_K \; bG_{\beta})_K
&\le 
\sum_K  c^{1/2} \delta_K^{1/2} \;\; c^{1/2} \Vert  \omega^{-1/2}G \Vert_K \cdot 
b\delta_K^{1/2} \Vert  \omega^{-1/2}G_{\beta} \Vert_K \\
&\le
\frac{1}{4}c\Vert  \omega^{-1/2}G \Vert^2+ \sum_K c   \delta_K\cdot  b^2\delta_K \Vert  \omega^{-1/2}G_{\beta} \Vert^2_K\\
&\le
\frac{1}{4}c\Vert  \omega^{-1/2}G \Vert^2+ \frac{1}{2}\sum_K   b^2\delta_K \Vert  \omega^{-1/2}G_{\beta} \Vert^2_K.
\end{align*}

If \eqref{eq:sigma-beta-eta-1} holds true and  $k$ is taken sufficiently
large and independent of $N$ and $\varepsilon$, 
then  we have
 $$
(\varepsilon+b^{2}\delta)| ((\omega^{-1})_{\beta}G,G_{\beta}) |+
\varepsilon|  ((\omega^{-1})_{\eta}G,G_{\eta}) |
+\sum_K  |( c \omega^{-1}G,\delta_K \; bG_{\beta})_K|
\le
\frac{3}{4}  \vvvert G  \vvvert^{2}_{\omega}.
$$
Considering \eqref{eq:norm and bilinear form}, we are done.

\end{proof}

\begin{lemma}\label{lem:2}
If  $1\ge \sigma_{\beta}\ge k (\varepsilon +\delta_M)$, with $k>0$ sufficiently
large and independent of $N$ and $\varepsilon$. Then for each mesh point $\boldsymbol{x}^{\ast}\in\Omega\setminus\Omega_{xy}$, we have
\begin{equation*}
\left|(\omega^{-1}G)(\boldsymbol{x}^{\ast})\right|
\le \frac{1}{16} \vvvert G  \vvvert^{2}_{\omega}+
\left\{
\begin{matrix}
CN^2\sigma_{\beta}& \text{if $\b{x}^*\in\Omega_s$}\\
CN\ln N& \text{if $\b{x}^*\in\Omega_x\cup\Omega_y$}
\end{matrix}
\right.
\end{equation*}
where $C$ is independent of $N$, $\varepsilon$ and $\boldsymbol{x}^{\ast}$.
\end{lemma}
\begin{proof}
First let $\boldsymbol{x}^{\ast}\in\Omega_{s}$. Let $\tau^{\ast}$ be the unique triangle that has $\boldsymbol{x}^{\ast}$ as its north-east corner. Then
\begin{align*}
\left| (\omega^{-1}G)(\boldsymbol{x}^{\ast})\right|
&\le CN\Vert G \Vert_{\tau^{\ast}}
\le
CN\underset{\tau^{\ast}}{\max}
\left|(\omega^{-1})^{-1/2}_{\beta}\right|\cdot
\Vert (\omega^{-1})^{1/2}_{\beta}G \Vert_{\tau^{\ast}}.
\end{align*}
Calculating $(\omega^{-1})^{-1}_{\beta}(\boldsymbol{x})$ explicitly, we see that
\begin{equation*}
(\omega^{-1})^{-1}_{\beta}(\boldsymbol{x})
\le
C\sigma_{\beta}\quad\forall\boldsymbol{x}\in\tau^{\ast}.
\end{equation*}
Thus the arithmetic-geometric mean inequality gives
\begin{equation*}
\left| (\omega^{-1}G)(\boldsymbol{x}^{\ast})\right|
\le
CN^2 \sigma_{\beta}+\frac{1}{16} \vvvert G  \vvvert^{2}_{\omega}.
\end{equation*}

\par
Next, let $\boldsymbol{x}^{\ast}\in\Omega_{x}$. (The case $\boldsymbol{x}^{\ast}\in\Omega_{y}$ is similar.) Write
$\boldsymbol{x}^{\ast}=(x_{i},y_{j})$.
Then
\begin{align*}
\left| (\omega^{-1}G)(\boldsymbol{x}^{\ast})\right|
&=
\left|G(\boldsymbol{x}^{\ast})\right|
=\left|\int^{1}_{x_{i}}G_{x}(t,y_{j})\mathrm{d}t\right|
\le
CH^{-1}_{y}\int^{1}_{x_{i}}\int^{y_{j+1}}_{y_{j}}
\left| G_{x}(t,y) \right|\mathrm{d}y\mathrm{d}t\\
&\le
CN(\varepsilon\ln N\cdot N^{-1})^{1/2}\Vert G_{x} \Vert_{\Omega_{x}}\le
CN^{1/2}\ln^{1/2}N \vvvert G  \vvvert\\
&\le
CN\ln N+\frac{1}{16} \vvvert G  \vvvert^{2}_{\omega}.
\end{align*}
Thus, we are done.
\end{proof}

\begin{lemma}\label{lem:4}
If 
\begin{equation}\label{eq: conditions-1}
\begin{split}
(\varepsilon+b^{2}\delta)^{1/2}\Vert \omega^{1/2}E_{\beta} \Vert
&+\hat{\varepsilon}^{1/2}\Vert \omega^{1/2}E_{\eta} \Vert
+\delta^{-1/2}_s\Vert \omega^{1/2}E \Vert_{\Omega_s}\\
&+\varepsilon^{-1/2}\Vert \omega^{1/2}E \Vert_{\Omega\setminus\Omega_s}
\le Ck^{-1/2}   \vvvert G  \vvvert_{\omega}.
\end{split}
\end{equation}
where $k>1$ sufficiently
large and independent of $N$ and $\varepsilon$,  then
\begin{equation*}
a_{MSD}(E,G)\le \frac{1}{16} \vvvert G  \vvvert^{2}_{\omega}
\end{equation*}
where $E=(\omega^{-1}G)^{I}-\omega^{-1}G$.
\end{lemma}
\begin{proof}
H\"{o}lder inequality gives
\begin{align*}
\left|a_{MSD}(E,G)\right|
\le &
(\varepsilon+b^{2}\delta)^{1/2}\Vert \omega^{1/2}E_{\beta} \Vert\cdot
(\varepsilon+b^{2}\delta)^{1/2}\Vert \omega^{-1/2}G_{\beta} \Vert\\
&+\hat{\varepsilon}^{1/2}\Vert \omega^{1/2}E_{\eta} \Vert\cdot
\hat{\varepsilon}^{1/2}\Vert \omega^{-1/2}G_{\eta} \Vert\\
&+C\Vert \omega^{1/2}E \Vert\cdot
\Vert \omega^{-1/2}G_{\beta} \Vert
+\Vert \omega^{1/2}E \Vert\cdot \Vert \omega^{-1/2}G \Vert.\\
&\le  \frac{1}{16} \vvvert G  \vvvert^{2}_{\omega},\quad 
\text{ apply \eqref{eq: conditions-1} and take $k$ sufficiently large} .
\end{align*}
\end{proof}
%
%

Set  $h_{\tau}=\max\{h_{x,\tau},h_{y,\tau}\}$. For $\forall \tau \in\mathcal{T}_N$ we have
\begin{align*}
\Vert \omega^{1/2} E \Vert_{\tau}
\le
C (\underset{\tau}{\max} \omega)^{1/2}
\cdot h^2_{\tau} \left( \Vert  (\omega^{-1}G)_{\beta\beta} \Vert_{\tau}+
\Vert  (\omega^{-1}G)_{\beta\eta} \Vert_{\tau} +\Vert  (\omega^{-1}G)_{\eta\eta} \Vert_{\tau} \right)
\end{align*}
and
\begin{align*}
&\Vert \omega^{1/2} E_{\beta} \Vert_{\tau}
\le  C (\underset{\tau}{\max} \omega)^{1/2}
\cdot h_{\tau} \left( \Vert  (\omega^{-1}G)_{\beta\beta} \Vert_{\tau}+
\Vert  (\omega^{-1}G)_{\beta\eta} \Vert_{\tau} +\Vert  (\omega^{-1}G)_{\eta\eta} \Vert_{\tau} \right),\\
&\Vert \omega^{1/2} E_{\eta} \Vert_{\tau}
\le  C (\underset{\tau}{\max} \omega)^{1/2}
\cdot h_{\tau} \left( \Vert  (\omega^{-1}G)_{\beta\beta} \Vert_{\tau}+
\Vert  (\omega^{-1}G)_{\beta\eta} \Vert_{\tau} +\Vert  (\omega^{-1}G)_{\eta\eta} \Vert_{\tau} \right),
\end{align*}
where we have used \cite[Corollary 3.1]{Linb1Styn2:2001-SDFEM}.  Clearly, the following estimate holds 
\begin{align*}
 &\Vert  (\omega^{-1}G)_{\beta\beta} \Vert_{\tau}+
\Vert  (\omega^{-1}G)_{\beta\eta} \Vert_{\tau} +\Vert  (\omega^{-1}G)_{\eta\eta} \Vert_{\tau} \\
\le&
C\left(  \Vert  (\omega^{-1})_{\beta\beta} G\Vert_{\tau} +\Vert  (\omega^{-1})_{\beta}G_{\beta} \Vert_{\tau}  \right) \\
&+C\left(  \Vert  (\omega^{-1})_{\beta\eta} G\Vert_{\tau} +\Vert  (\omega^{-1})_{\beta}G_{\eta} \Vert_{\tau} +\Vert  (\omega^{-1})_{\eta}G_{\beta} \Vert_{\tau} \right)\\
&+C\left(  \Vert  (\omega^{-1})_{\eta\eta} G\Vert_{\tau} +\Vert  (\omega^{-1})_{\eta}G_{\eta} \Vert_{\tau}  \right).
\end{align*}

Set $\alpha_{\omega}:=(\underset{\tau}{\min} \omega)^{-1/2}$. Lemma 4.1 in \cite{Linb1Styn2:2001-SDFEM} yields
\begin{equation}\label{eq:energy-G-b}
\begin{split}
\Vert  (\omega^{-1})_{\beta\beta} G\Vert_{\tau}
&\le
C\alpha_{\omega}\cdot
\sigma^{-3/2}_{\beta}\cdot \Vert  (\omega^{-1})^{1/2}_{\beta} G\Vert_{\tau},\\
\Vert  (\omega^{-1})_{\beta}G_{\beta} \Vert_{\tau}
&\le
C\alpha_{\omega}\cdot \sigma^{-1}_{\beta}\cdot \Vert  \omega^{-1/2}G_{\beta} \Vert_{\tau},\\
\Vert  (\omega^{-1})_{\beta\eta}G  \Vert_{\tau}
&\le
C\alpha_{\omega}\cdot
\sigma^{-1/2}_{\beta}\sigma^{-1}_{\eta}\cdot\Vert  (\omega^{-1})^{1/2}_{\beta} G\Vert_{\tau},\\
\Vert  (\omega^{-1})_{\eta }G_{ \beta}  \Vert_{\tau}
&\le
C\alpha_{\omega}\cdot
\sigma^{-1}_{\eta}\cdot\Vert  \omega^{-1/2}G_{\beta} \Vert_{\tau},\\
\Vert  (\omega^{-1})_{\eta\eta}G  \Vert_{\tau}
&\le
C\alpha_{\omega}\cdot
\sigma^{-2}_{\eta}\cdot\Vert   \omega^{-1/2} G\Vert_{\tau}.
\end{split}
\end{equation}

Similarly, for $\tau\subset\Omega_s$ we have
\begin{equation}
\Vert  (\omega^{-1})_{\eta}G_{\eta} \Vert_{\tau} 
\le
C\alpha_{\omega}\cdot
\left\{
\begin{matrix}
\sigma^{-1}_{\eta}N\cdot\Vert   \omega^{-1/2} G \Vert_{\tau}& \text{ if
 $\hat{\varepsilon}_s\le N^{-2}$}\\
\sigma^{-1}_{\eta}\hat{\varepsilon}^{-1/2}_s\cdot \hat{\varepsilon}^{1/2}_s\Vert   \omega^{-1/2} G_{\eta} \Vert_{\tau} &\text{ if $\hat{\varepsilon}_s\ge N^{-2}$}
\end{matrix}
\right.
\end{equation}
and for $ \tau\subset\Omega\setminus\Omega_s$ 
\begin{equation}
\Vert  (\omega^{-1})_{\eta}G_{\eta} \Vert_{\tau} 
\le
C\alpha_{\omega}\cdot
\varepsilon^{-1/2}\sigma^{-1}_{\eta}\cdot\varepsilon^{1/2}\Vert   \omega^{-1/2} G_{\eta}\Vert_{\tau}.
\end{equation}

For $\tau\subset\Omega_s$, we have
\begin{align}
\Vert  (\omega^{-1})_{\beta }G_{ \eta}  \Vert_{\tau}
&\le
C \underset{\tau}{\max} (\omega^{-1})_{\beta }\cdot  \Vert G_{ \eta}  \Vert_{\tau}
\le
C \underset{\tau}{\max} (\omega^{-1})_{\beta }\cdot N \cdot \Vert G  \Vert_{\tau} \nonumber\\
&\le
C N\left( \underset{\tau}{\max} (\omega^{-1})_{\beta } \right)^{1/2}
\left( \underset{\tau}{\min} (\omega^{-1})_{\beta } \right)^{1/2}
\cdot     \Vert G  \Vert_{\tau}\nonumber\\
 &\le
C\alpha_{\omega} \cdot N \sigma^{-1/2}_{\beta}\cdot     \Vert (\omega^{-1})^{1/2}_{\beta } G  \Vert_{\tau}
\end{align}
where we have used   \cite[Lemma 4.1 (vii)]{Linb1Styn2:2001-SDFEM}.
For $\tau\subset\Omega\setminus\Omega_s$, we have
\begin{align}
&\Vert  (\omega^{-1})_{\beta }G_{ \eta}  \Vert_{\tau}
\le
C\alpha_{\omega}\cdot
 \varepsilon^{-1/2}\sigma^{-1}_{\beta}\cdot \varepsilon^{1/2}\Vert  \omega^{-1/2}G_{\eta} \Vert_{\tau}\label{eq:energy-G-e}.
\end{align}

Set 
\begin{equation}\label{eq:simga-eta-hat-varepsilon}
\sigma_{\eta}(\hat{\varepsilon})=
\left\{
\begin{matrix}
\sigma^{-1}_{\eta}N& \text{ if
 $\hat{\varepsilon}_s\le N^{-2}$}\\
\sigma^{-1}_{\eta}\hat{\varepsilon}^{-1/2}_s &\text{ if $\hat{\varepsilon}_s\ge N^{-2}$}
\end{matrix}
\right..  
\end{equation}
From \eqref{eq:energy-G-b}--\eqref{eq:energy-G-e} and \cite[Lemma 4.1 (iii)]{Linb1Styn2:2001-SDFEM}, we have 
\begin{align}
&\Vert \omega^{1/2}E \Vert_{\Omega_{s}}+N^{-1}\left( \Vert \omega^{1/2}E_{\beta} \Vert_{\Omega_{s}} +\Vert \omega^{1/2}E_{\eta} \Vert_{\Omega_{s}}  \right) \nonumber\\
\le&
C N^{-2} \big(  \sigma^{-3/2}_{\beta}+\sigma^{-1}_{\beta}\delta^{-1/2}_{s}+\sigma^{-1/2}_{\beta}\sigma^{-1}_{\eta}
+\sigma^{-1}_{\eta}\delta^{-1/2}_{s}+\sigma^{-2}_{\eta}+
\sigma_{\eta}(\hat{\varepsilon})
+\sigma^{-1/2}_{\beta}N \big)  \vvvert G  \vvvert_{\omega,\Omega_s}\nonumber \\
&\Vert \omega^{1/2}E \Vert_{\Omega\setminus\Omega_s } +N^{-1}\left( \Vert \omega^{1/2}E_{\beta} \Vert_{\Omega\setminus \Omega_{s} } +\Vert \omega^{1/2}E_{\eta} \Vert_{\Omega\setminus \Omega_{s} }  \right) \label{eq:omega-E-not s}\\
\le&
C N^{-2} \big(  \sigma^{-3/2}_{\beta}+\sigma^{-1}_{\beta}\varepsilon^{-1/2}+\sigma^{-1/2}_{\beta}\sigma^{-1}_{\eta}
+\sigma^{-1}_{\eta}\varepsilon^{-1/2}+\sigma^{-2}_{\eta} \big)  \vvvert G  \vvvert_{\omega,\Omega\setminus\Omega_s}. \nonumber
\end{align}

To make sure that \eqref{eq: conditions-1}  holds,  we should have
\begin{align}
&\Vert \omega^{1/2}E_{\beta} \Vert_{\Omega_{s}}  \le C k^{-1/2} (\varepsilon+\delta_s)^{-1/2}  \vvvert G  \vvvert_{\omega,\Omega_s  },\label{eq:condition-2-1}\\
&\Vert \omega^{1/2}E_{\eta} \Vert_{\Omega_{s}}  \le C k^{-1/2} \hat{\varepsilon}^{-1/2}_s  \vvvert G  \vvvert_{\omega,\Omega_s  },\label{eq:condition-2-2}\\
&\Vert \omega^{1/2}E_{\beta} \Vert_{\Omega\setminus \Omega_{s} } +\Vert \omega^{1/2}E_{\eta} \Vert_{\Omega\setminus \Omega_{s} }
\le Ck^{-1/2} \varepsilon^{-1/2}  \vvvert G  \vvvert_{\omega,\Omega\setminus \Omega_{s}   },\label{eq:condition-2-3}\\
&\Vert \omega^{1/2}E \Vert_{\Omega_{s}} \le C k^{-1/2} \delta_s^{ 1/2} \vvvert G  \vvvert_{\omega,\Omega_s  },\label{eq:condition-2-4}
\end{align}
and
\begin{equation}\label{eq:condition-2-5}
\Vert \omega^{1/2}E \Vert_{ \Omega\setminus \Omega_{s}  } \le C k^{-1/2} \varepsilon^{ 1/2} \vvvert G  \vvvert_{\omega, \Omega\setminus \Omega_{s}  }.
\end{equation}
To ensure that  \eqref{eq:condition-2-1} holds true,  we can set
\begin{equation}\label{eq:sigma-beta-eta-2}
\begin{split}
&\sigma_{\beta}\ge k N^{-1},\\
&\sigma_{\eta}\ge k N^{-3/4},\quad 
\sigma_{\eta}\ge  \sigma_{\eta}^*:=
\left\{
\begin{matrix}
kN^{-1/2} &\text{if $\hat{\varepsilon}_s\le N^{-2}$}\\
k\hat{\varepsilon}^{-1/2}_sN^{-3/2} &\text{if $\hat{\varepsilon}_s\ge N^{-2}$}
\end{matrix}
\right.
.
\end{split}
\end{equation}
Assume that \eqref{eq:sigma-beta-eta-2} holds true, then \eqref{eq:condition-2-1}--\eqref{eq:condition-2-4} hold and we have
\begin{equation}\label{eq:omega-E-not s}
\Vert \omega^{1/2}E \Vert_{\Omega\setminus\Omega_s } \le Ck^{-1} \varepsilon^{-1/2}  N^{-1}  \vvvert G  \vvvert_{\omega  }.
\end{equation}

%

Next, we are to obtain sharper bounds for $\Vert \omega^{1/2}E \Vert_{\Omega\setminus\Omega_s }$.  Following the techniques of (see \cite[Lemma 4.4]{Linb1Styn2:2001-SDFEM}), we have
\begin{equation}\label{eq:omega-E}
(\omega^{1/2}E)(\boldsymbol{x})=\int_{\boldsymbol{x}}^{\Gamma(\boldsymbol{x})}(\omega^{1/2}E)_{\eta}\mathrm{d}s
\end{equation}
where $\boldsymbol{x}\in\Omega\setminus\Omega_{s}$, $\Gamma(\boldsymbol{x})\in \partial\Omega$ satisfies
  $(\boldsymbol{x}-\Gamma(\boldsymbol{x}))\cdot \beta=0$ and the following condition:
 \begin{equation*}
 |\boldsymbol{x}-\Gamma(\boldsymbol{x})|=\underset{y}{\min} |\boldsymbol{x}-\boldsymbol{y}|,  \text{where $ \boldsymbol{y}\in \partial\Omega$ and $(\boldsymbol{x}-\boldsymbol{y})\cdot \beta=0$, }
 \end{equation*}
From \eqref{eq:omega-E}, we have
\begin{align*}
&\Vert \omega^{1/2}E \Vert^{2}_{\Omega_x\cup\Omega_{xy} }=
\int^1_{1-\lambda_x}\int^{1  }_{0}
\left[\int_{\boldsymbol{x}}^{\Gamma(\boldsymbol{x})}(\omega^{1/2}E)_{\eta}\mathrm{d}s\right]^{2}\mathrm{d}y\mathrm{d}x\\
&\le
C\lambda^{2}_{x}
\left\{\Vert(\omega^{1/2})_{\eta}E\Vert^{2}_{ \Omega\setminus\Omega_s }+
 \Vert\omega^{1/2}E_{\eta}\Vert^{2}_{  \Omega\setminus\Omega_s }
\right\}\\
&\le
C\varepsilon^{2}\ln^{2}N
\left( \sigma^{-2}_{\eta}\Vert \omega^{1/2}E \Vert^{2}_{ \Omega\setminus\Omega_s }+
 \Vert\omega^{1/2}E_{\eta}\Vert^{2}_{ \Omega\setminus\Omega_s }
\right).
\end{align*}
Similar argument holds for $\Omega_y$. Thus, we have
\begin{equation}\label{eq:E-Omega-x-y-xy}
\Vert \omega^{1/2}E \Vert^{2}_{\Omega\setminus\Omega_s }\le
C\varepsilon^{2}\ln^{2}N
\left( \sigma^{-2}_{\eta}\Vert \omega^{1/2}E \Vert^{2}_{ \Omega\setminus\Omega_s }+
 \Vert\omega^{1/2}E_{\eta}\Vert^{2}_{ \Omega\setminus\Omega_s }
\right).
\end{equation}
To make sure that \eqref{eq:condition-2-5} holds,  according to \eqref{eq:E-Omega-x-y-xy}  the following  estimates should hold
\begin{align}
&\varepsilon^{2}\ln^{2}N
 \sigma^{-2}_{\eta}\Vert \omega^{1/2}E \Vert^{2}_{ \Omega\setminus\Omega_s } \le C k^{-1} \varepsilon \vvvert G  \vvvert^2_{\omega, \Omega\setminus \Omega_{s}  }\nonumber\\
&\Vert\omega^{1/2}E_{\eta}\Vert^{2}_{ \Omega\setminus\Omega_s }
 \le C k^{-1} \varepsilon^{-1} \ln^{-2}N \vvvert G  \vvvert^2_{\omega, \Omega\setminus \Omega_{s}  }\label{eq:omega-E-eta-x-y-xy}
\end{align}
From \eqref{eq:omega-E-eta-x-y-xy} and \eqref{eq:omega-E-not s}, we can set
\begin{equation}\label{eq:sigma-beta-eta-3}
\sigma_{\beta} \ge k N^{-1}\ln N,\quad  \sigma_{\eta}\ge k N^{-1}\ln  N,\quad  \sigma_{\eta}\ge k \varepsilon^{1/4}N^{-1/2}\ln^{1/2}  N.
\end{equation}

Assume that \eqref{eq:sigma-beta-eta-2} and  \eqref{eq:sigma-beta-eta-3} hold true,  substituting \eqref{eq:omega-E-not s} and \eqref{eq:omega-E-eta-x-y-xy} into the right-hand side of
\eqref{eq:E-Omega-x-y-xy}, we have
\begin{align}
\Vert \omega^{1/2}E \Vert^{2}_{\Omega\setminus\Omega_s }&\le
Ck^{-1}\varepsilon 
\vvvert G  \vvvert^{2}_{\omega}.
\end{align}

Thus, we have the following lemma.
\begin{lemma}\label{lemma derivative of E}
Assume  that $\sigma_{\beta}$ and $\sigma_{\eta}$ satisfy \eqref{eq:sigma-beta-eta-2} and \eqref{eq:sigma-beta-eta-3}, where $k>1$ sufficiently
large and independent of $N$ and $\varepsilon$, then we have
\begin{align*}
&\Vert \omega^{1/2}E \Vert_{ \Omega_s } \le Ck^{-1/2}N^{-1/2} \vvvert G   \vvvert_{\omega,\Omega_s },\\
&\Vert \omega^{1/2}E \Vert_{\Omega\setminus\Omega_s }
\le
Ck^{-1/2}\varepsilon^{1/2} \vvvert G  \vvvert_{\omega,\Omega\setminus\Omega_s}
\end{align*}
and
\begin{align*}
\Vert \omega^{1/2}E_{\beta}\Vert_{\Omega_{s}}+\Vert \omega^{1/2}E_{\eta}\Vert_{\Omega_{s}}&\le Ck^{-1/2}N^{1/2} \vvvert G  \vvvert_{\omega,\Omega_s },\\
\Vert \omega^{1/2}E_{\beta}  \Vert_{\Omega \setminus\Omega_s }+\Vert \omega^{1/2}E_{\eta}  \Vert_{\Omega \setminus\Omega_s }&\le Ck^{-1/2}\varepsilon^{-1/2}\ln^{-1}N  \vvvert G   \vvvert_{\omega,\Omega\setminus\Omega_s}.
\end{align*}
\end{lemma}

\begin{theorem}
Assume  that $\sigma_{\beta}$ and $\sigma_{\eta}$ satisfy \eqref{eq:sigma-beta-eta-1}, \eqref{eq:sigma-beta-eta-2} and \eqref{eq:sigma-beta-eta-3}, where
$k$ is chosen so  that Lemmas \ref{lem:1}, \ref{lem:2} and \ref{lem:4} hold. Then for
$\boldsymbol{x}^{\ast}\in \Omega\setminus\Omega_{xy}$, we have
$$
 \Vert G \Vert_{MSD} \le \sqrt{8}\vvvert G \vvvert_{\omega} \le  
 \left\{
\begin{matrix}
CN\sigma^{1/2}_{\beta} &\text{ if $\b{x}^*\in\Omega_s$ }\\
CN^{1/2}\ln^{1/2} N &\text{ if $\b{x}^*\in\Omega_x\cup\Omega_y$ }
\end{matrix}
\right.
$$
\end {theorem}
\begin{proof}
The readers are referred to \cite[Theorem 4.1]{Linb1Styn2:2001-SDFEM} for the estimate $ \Vert G \Vert_{MSD} \le \sqrt{8}\vvvert G \vvvert_{\omega}$.
Considering \eqref{eq:norm and bilinear form}, \eqref{eq:idea-weighted estimate}  and  Lemmas \ref{lem:1}--\ref{lem:4}, we obtain
\begin{align*}
\frac{1}{4} \vvvert G \vvvert^{2}_{\omega} \le &
a_{MSD}(\omega^{-1}G,G) = a_{MSD}(E,G)+(\omega^{-1}G)(\boldsymbol{x}^{\ast})\\
\le &\frac{1}{8} \vvvert G \vvvert^{2}_{\omega}+
\left\{
\begin{matrix}
CN^2\sigma_{\beta} &\text{ if $\b{x}^*\in\Omega_s$ }\\
CN\ln N &\text{ if $\b{x}^*\in\Omega_x\cup\Omega_y$ }
\end{matrix}
\right.
.
\end{align*}
Thus we are done.
\end{proof}

From \eqref{eq:sigma-beta-eta-1}, \eqref{eq:sigma-beta-eta-2} and \eqref{eq:sigma-beta-eta-3} we have
$$
\sigma_{\beta}\ge k N^{-1}\ln N,
\quad 
\sigma_{\eta}\ge k\max\{ \hat{\varepsilon }^{1/2}_M,N^{-3/4},\sigma_{\eta}^*, N^{-1}\ln  N,\varepsilon^{1/4}N^{-1/2}\ln^{1/2}  N\}
$$
where $\sigma_{\eta}^*=
\left\{
\begin{matrix}
kN^{-1/2} &\text{if $\hat{\varepsilon}_s\le N^{-2}$}\\
k\hat{\varepsilon}^{-1/2}_sN^{-3/2} &\text{if $\hat{\varepsilon}_s\ge N^{-2}$}
\end{matrix}
\right.$. Note that if $N\ge 4$,  then  $N^{1/2}\ge \ln N$ and
$$
N^{-1/2} \ge N^{-1}\ln N,\quad 
N^{-1/2} \ge \varepsilon^{1/4} N^{-1/2}\ln^{1/2}  N\quad \text{ for $N\ge 4$}.
$$

Now, we consider the case of $\hat{\varepsilon}=\varepsilon$. Clearly, $\hat{\varepsilon }_M=\hat{\varepsilon}_s=\varepsilon$.
\begin{itemize}
\item
 If $\varepsilon\le N^{-2}$, then $\sigma_{\eta}^*=kN^{-1/2}$  and we have
 \begin{align*}
&N^{-1/2}\ge \varepsilon^{1/2},\quad N^{-1/2}\ge  N^{-3/4},\\
&N^{-1/2} \ge N^{-1}\ln N\quad \text{for $N\ge 4$},\\
&N^{-1/2} \ge \varepsilon^{1/4}N^{-1/2}\ln ^{1/2} N\quad \text{for $N\ge 4$}.
\end{align*}

\item
If $N^{-2}<\varepsilon\le N^{-3/2}$, then $\sigma_{\eta}^*=k\varepsilon^{-1/2}N^{-3/2}$ and we have
\begin{align*}
&\varepsilon^{-1/2}N^{-3/2}\ge \varepsilon^{1/2},\quad \varepsilon^{-1/2}N^{-3/2}\ge  N^{-3/4}.
\end{align*}
It means that $\sigma_{\eta}\ge   k\max\{ \varepsilon^{-1/2}N^{-3/2},N^{-1}\ln N,\varepsilon^{1/4}N^{-1/2}\ln ^{1/2} N \} $. Note that if $N^{-2}<\varepsilon\le N^{-3/2}$,
\begin{align*}
 &N^{-1/2} >\varepsilon^{-1/2}N^{-3/2}\\
  &N^{-1/2} >N^{-1}\ln N\quad \text{for $N\ge 4$},\\
    &N^{-1/2} >\varepsilon^{1/4}N^{-1/2}\ln ^{1/2} N \quad \text{for $N\ge 4$}.
\end{align*}
\item
If $N^{-3/2}<\varepsilon\le N^{-1}$, then $\sigma_{\eta}^*=k\varepsilon^{-1/2}N^{-3/2}$ and we have
\begin{align*}
&\varepsilon^{1/2}\ge  \varepsilon^{-1/2}N^{-3/2},\quad \varepsilon^{1/2}\ge  N^{-3/4}.
\end{align*}
It means that $\sigma_{\eta}\ge   k\max\{ \varepsilon^{1/2}, N^{-1}\ln N,\varepsilon^{1/4}N^{-1/2}\ln ^{1/2} N \} $. 
Note that if $N^{-3/2}<\varepsilon\le N^{-1}$,
\begin{align*}
 &N^{-1/2} >\varepsilon^{1/2}\\
  &N^{-1/2} >N^{-1}\ln N\quad \text{for $N\ge 4$},\\
    &N^{-1/2} >\varepsilon^{1/4}N^{-1/2}\ln ^{1/2} N \quad \text{for $N\ge 4$}.
\end{align*}
\end{itemize}
These mean that \emph{in the case of $\hat{\varepsilon}=\varepsilon$ we can set $\sigma_{\eta}\ge k N^{-1/2}$}.\\
~\\

If $
\hat{\varepsilon}=
\left\{
\begin{matrix}
\tilde{\varepsilon} & \text{if $\b{x}\in \Omega_s$},\\
\varepsilon & \text{otherwise},
\end{matrix}
\right.
$
where $\tilde{\varepsilon}:=\max \left( \varepsilon, N^{-3/2} \right)$. Clearly, $\hat{\varepsilon }_M=\hat{\varepsilon}_s=\tilde{\varepsilon}$. Note that
$\sigma_{\eta}^*=k\tilde{\varepsilon}^{-1/2}N^{-3/2}$ for $\hat{\varepsilon}_s \ge N^{-3/2}$.
\begin{itemize}
\item
 If $\varepsilon\le N^{-3/2}$, then $\hat{\varepsilon}_M=\tilde{\varepsilon}=N^{-3/2}$  and we have
 \begin{align*}
&\hat{\varepsilon}^{1/2}_M=\tilde{\varepsilon}^{1/2}= N^{-3/4},\quad \tilde{\varepsilon}^{-1/2}N^{-3/2}= N^{-3/4}.
\end{align*}
It means that $\sigma_{\eta}\ge   k\max\{ \tilde{\varepsilon}^{1/2}, N^{-1}\ln N,\varepsilon^{1/4}N^{-1/2}\ln ^{1/2} N \} $.   Note that
\begin{align*}
&\tilde{\varepsilon}^{1/2}\ln^{1/2}N >N^{-1}\ln N\quad \text{for $N\ge 4$},\\
    &\tilde{\varepsilon}^{1/2}\ln^{1/2}N >\varepsilon^{1/4}N^{-1/2}\ln ^{1/2} N \quad \text{for $N\ge 4$}.
\end{align*}
\item
If $N^{-3/2}<\varepsilon\le N^{-1}$, then $\tilde{\varepsilon}=\varepsilon>N^{-3/2}$ and we have
\begin{align*}
&\tilde{\varepsilon}^{1/2}>\varepsilon^{-1/2}N^{-3/2},\\
&\tilde{\varepsilon}^{1/2}\ln^{1/2}N\ge N^{-1}\ln  N \quad \text{for $N\ge 4$}\\
&\tilde{\varepsilon}^{1/2}\ln^{1/2}N\ge \varepsilon^{1/4}N^{-1/2}\ln ^{1/2} N \quad \text{for $N\ge 4$}
\end{align*}
\end{itemize}
These mean that   \emph{we can take $\sigma_{\eta}\ge k \tilde{\varepsilon}^{1/2}\ln^{1/2}N$}. 

From the above derivations, we have the following remark.
\begin{remark}
If we set $\hat{\varepsilon}=\varepsilon$,  we can take
$$
\sigma_{\beta}=k N^{-1}\ln N,\quad \sigma_{\eta}=
kN^{-1/2}.
$$

If we set $\hat{\varepsilon}$ as in  \cite[pg 463]{Linb1Styn2:2001-SDFEM}, i.e.,
$$
\hat{\varepsilon}=
\left\{
\begin{matrix}
\tilde{\varepsilon} & \text{if $\b{x}\in \Omega_s$},\\
\varepsilon & \text{otherwise},
\end{matrix}
\right.
$$
where $\tilde{\varepsilon}:=\max \left( \varepsilon, N^{-3/2} \right)$,
we can define
$$
\sigma_{\beta}=k N^{-1}\ln N,\quad \sigma_{\eta}=k\tilde{\varepsilon}^{1/2}\ln^{1/2}N.
$$
It seems that the definition of $\sigma_{\beta}$ in \cite{Linb1Styn2:2001-SDFEM}  is not proper. 
\end{remark}

\end{document}